\documentclass[11pt,reqno]{amsart}

\numberwithin{equation}{section}

\usepackage{color,graphics,epsfig}

\newcommand{\calE}{\mathcal{E}}

\newcommand{\mC}{\mathbb{C}}
\newcommand{\mD}{\mathbb{D}}

\newcommand{\mN}{\mathbb{N}}

\newcommand{\mR}{\mathbb{R}}

\newcommand{\mZ}{\mathbb{Z}}

\newtheorem{theorem}{Theorem}[section]
\newtheorem{lemma}[theorem]{Lemma}
\newtheorem{corollary}[theorem]{Corollary}
\newtheorem{proposition}[theorem]{Proposition}

\theoremstyle{definition}

\newtheorem{example}[theorem]{Example}

\theoremstyle{definition}

\theoremstyle{definition}

\begin{document}



\title[Generators for rings of distributions]{Generators for rings of
  compactly supported distributions}

\author{Sara Maad Sasane}
\address{Department of Mathematics, Stockholm University, Sweden.}
\email{maad@math.su.se}

\author{Amol Sasane}
\address{Department of Mathematics, Royal Institute of Technology,
    Stockholm, Sweden.}
\email{sasane@math.kth.se}

\begin{abstract}
  Let $C$ denote a closed convex cone $C$ in $\mR^d$ with apex at $0$.
  We denote by $\calE'(C)$ the set of distributions having compact
  support which is contained in $C$. Then $\calE'(C)$ is a ring with
  the usual addition and with convolution. We give a necessary and
  sufficient analytic condition on $\widehat{f}_1,\dots,
  \widehat{f}_n$ for $f_1,\dots ,f_n \in \calE'(C)$ to generate the
  ring $\calE'(C)$.  (Here $\widehat{\;\cdot\;}$ denotes
  Fourier-Laplace transformation.) This result is an application of a
  general result on rings of analytic functions of several variables
  by H\"ormander.  En route we answer an open question posed by Yutaka
  Yamamoto.
\end{abstract}

\maketitle

\section{Introduction}

Let $R$ be a commutative ring with identity. Elements $a_1, \dots,
a_n$ of $R$ are said to {\em generate} $R$ if the ideal generated by
$a_1,\dots, a_n$ is equal to $R$, or equivalently, if there exist
$b_1,\dots,b_n$ such that $a_1 b_1+ \dots+ a_n b_n=1$.

For instance, if $R=H^\infty(\mD)$, the set of all bounded and
holomorphic functions on the open unit disc $\mD$ centered at $0$ in
$\mC$, then the corona theorem says that $f_1, \dots, f_n \in
H^\infty(\mD)$ generate $H^\infty(\mD)$ iff there exists a $C>0$ such
that $ |f_1(z)|+\dots+ |f_n(z)|>C$ for all $z\in \mD$; see \cite{Car}.

In this note we address this question when the ring $R$ consists of
compactly supported distributions. 

Let $C$ denote a closed convex cone $C$ in $\mR^d$ with apex at $0$.
Recall that a {\em convex cone} is a subset of $\mR^d$ with the
following properties:
\begin{enumerate}
\item If $x, y\in C$, then $x+y \in C$. 
\item If $x\in C$ and $t>0$, then $tx\in C$.  
\end{enumerate}
Let $\calE'(C)$ be the set consisting of all distributions having a
compact support which is contained in $C$. Then $\calE'(C)$ is a
commutative ring with the usual addition of distributions and the
operation of convolution.  The Dirac delta distribution $\delta$
supported at $0$ serves as an identity in the ring $\calE'(C)$. 

Recall that a distribution $f$ with compact support has a finite
order, and its Fourier-Laplace transform is an entire function given
by
$$
\widehat{f}(z)=\langle f, e^{-iz\cdot} \rangle ,\quad z\in \mC^d.
$$
We use the notation $\|\cdot \|$ for the usual Euclidean $2$-norm in
$\mC^d$. The same notation is also used for the Euclidean norm in
$\mR^d$. 

The {\em supporting function} of a convex, compact set $K$ ($\subset
\mR^d$) is defined by
$$
H_K(\xi)= \sup_{x\in K} \langle x, \xi\rangle, \quad \xi \in \mR^d.
$$ 
Our main result is the following:

\begin{theorem}
\label{theorem_1} 
Let $C$ denote a closed convex cone in $\mR^d$ with apex at $0$,
and $H$ denote the supporting function of the compact convex set 
$$
B:=C\cap \{ x\in \mR^d: \|x\|\leq 1\},
$$ 
that is, $H(\xi)=\displaystyle \sup_{x\in B} \langle \xi, x\rangle$.

Let $f_1, \dots, f_n\in \calE'(C)$. There exist $g_1,\dots,
g_n\in\calE'(C)$ such that
$$
f_1 \ast g_1 + \dots + f_n \ast g_n =\delta
$$
iff there are positive constants $C,N , M$ such that 
\begin{equation}
\label{CC}
\textrm{for all }z\in\mC^d\textrm{, } 
|\widehat{f}_1(z)|+\dots+|\widehat{f}_n(z)|\geq 
C (1+ \|z\|^2)^{-N} e^{-M H(\textrm{\em Im}(z))}.
\end{equation} 
\end{theorem}

Theorem~\ref{theorem_1}, in the case when $d=1$ and $C=\mR$ was known; see
\cite{Mei}.

\section{Proof of the main result}

We will show that our main result follows from the main result
given in \cite{Hor}. We will use the Payley-Wiener-Schwartz theorem,
which is recalled below.

\begin{proposition}
[Payley-Wiener-Schwartz]
Let $K$ be a convex compact subset of $\mR^d$ with supporting function
$H$. If $u$ is a distribution with support contained in $K$, then
there exists a positive $N$ such that
\begin{equation}
\label{PWS}
\textrm{for all }z\in\mC^d\textrm{, }
\widehat{u}(z) \leq C (1+\|z\|^2)^N e^{H(\textrm{\em Im}(z))}.
\end{equation}
Conversely, every entire analytic function in $\mC^d$ satisfying an
estimate of the form \eqref{PWS} is the Fourier-Laplace transform of a
distribution with support contained in $K$.
\end{proposition}
\begin{proof} See for instance \cite[Theorem~7.3.1]{Hor83}. The only
  difference is that we have the term $(1+\|z\|^2)^N$ instead of
  $(1+\|z\|)^N$ in the estimate \eqref{PWS}, which follows from the
  observation that $1+\|z\|^2 \leq (1+\|z\|)^2 \leq 2(1+\|z\|^2)$ for
  every $z\in \mC^d$ (and by replacing $N/2$ by $N$).
\end{proof}

We also recall the main result from H\"ormander \cite[Theorem 1, p.
943]{Hor}, which we will use.

Let $p$ be a nonnegative function defined in $\mC^d$. Let $A_p$ denote
the set of all entire functions $F:\mC^d\rightarrow \mC$ such that
there exist positive constants $C_1$ and $C_2$ (which in general
depend on $F$) such that
$$
\textrm{for all }z\in\mC^d\textrm{, } |F(z)| \leq C_1 e^{C_2 p(z)}.
$$
It is clear that $A_p$ is a ring with the usual pointwise operations.

\begin{proposition}[H\"ormander]
  \label{prop_Hor}Let $p$ be a nonnegative plurisubharmonic function
  in $\mC^d$ such that
\begin{enumerate}
\item all polynomials belong to $A_p$
\item there exist nonnegative $K_1, K_2, K_3, K_4$ such that
  whenever $z,\zeta \in \mC^d$ satisfy $\|z-\zeta\|\leq e^{-K_1
    p(z)-K_2}$, there holds that $p(\zeta)\leq K_3 p(z)+K_4$.
\end{enumerate}
If there exist positive constants $C_1, C_2 $ such that 
\begin{equation}
\label{CCHor}
\textrm{for all }z\in\mC^d\textrm{, }
|F_1(z)|+\dots+ |F_n(z)|\geq C_1 e^{-C_2 p(z)},
\end{equation}
then  $F_1,\dots, F_n\in A_p$ generate $A_p$.
\end{proposition}

\begin{lemma}
\label{lemma_tech_1}
Let $C$ denote a closed convex cone in $\mR^d$ with apex at $0$, and
$H$ denote the supporting function of the compact convex set 
$$
B:=C\cap \{ x\in \mR^d: \|x\|\leq 1\},
$$ 
that is, $H(\xi)= \displaystyle \sup_{x\in B} \langle \xi, x\rangle$.

Let $p(z):= \log (1+ \|z\|^2) + H(\textrm{\em Im}(z))$. 
Then we have the following:
\begin{enumerate}
\item $p$ is nonnegative and subharmonic.$\phantom{\widehat{\calE'(C)}}$
\item $A_p= \widehat{\calE'(C)}$.
\item $A_p$ contains the polynomials.$\phantom{\widehat{\calE'(C)}}$
\item $\phantom{\widehat{\calE'(C)}}\!\!\!\!\!\!\!\!\!\!\!\!\!\!\!\!$ There exist nonnegative $K_1, K_2, K_3, K_4$ such that whenever
  $z,\zeta \in \mC^d$ satisfy $\|z-\zeta\|\leq e^{-K_1 p(z)-K_2}$, there
  holds that $p(\zeta)\leq K_3 p(z)+K_4$. (That is, condition (2) of
  Proposition~\ref{prop_Hor} is satisfied.)
\end{enumerate}
\end{lemma}
\begin{proof} (1) Clearly $p$ is nonnegative. Also, the complex
  Hessian at $z$ of the map $z\mapsto \log (1+\|z\|^2)$ is easily seen
  to be
$$
F(z):=\frac{1}{1+\|z\|^2} I-\frac{1}{(1+\|z\|^2)^2}zz^*.
$$
So for $w\in \mC^d$, we have that 
\begin{eqnarray*}
w^*F(z) w&=&\frac{1}{1+\|z\|^2} \|w\|^2-\frac{1}{(1+\|z\|^2)^2}|w^*z|^2\\
&=&\frac{\|w\|^2+\|w\|^2\|z\|^2-|w^*z|^2}{(1+\|z\|^2)^2}\geq 0
\end{eqnarray*}
by the Cauchy-Schwarz inequality. So the map $z\mapsto \log (1+|z|^2)$
is plurisubharmonic; see \cite[Proposition~4.9, p.88]{Ran}.

We will use the fact that a map $\varphi:\mC^d\rightarrow \mR$ that
depends only on the imaginary part of the variable is plurisubharmonic
iff the map is convex; see \cite[E.4.8, p.92]{Ran}. The supporting
function $H_K$ of any convex compact set $K$ satisfies the properties
that
$$
H_K(\xi+ \eta)\leq H_K(\xi)+ H_K(\eta), \quad H_K(t\xi)=t H_K(\xi)
$$ 
for all $\xi,\eta\in \mR^d$ and $t\geq 0$. It is then clear that $H_K$
is a convex function. In particular our $H$ (the supporting function
of $B$) is convex too. Thus $z\mapsto H(\textrm{Im}(z))$ is
plurisubharmonic.  Consequently, $p$, which is the sum of the
plurisubharmonic maps $ z\mapsto \log (1+\|z\|^2)$ and $z\mapsto
H(\textrm{Im}(z))$, is plurisubharmonic as well; \cite[p.88]{Ran}.

\medskip

\noindent (2) Suppose that $f\in \calE'(C)$ has support contained in
the compact set $K$ contained in $C$. Then by the
Payley-Wiener-Schwartz Theorem, there exist positive $C$, $N$, $M$
such that
$$
|\widehat{f}(z)|\leq C(1+\|z\|^2)^N e^{H_K(\textrm{Im}(z))}.
$$
Let $\epsilon>0$ be such that $\epsilon K \subset B$. Then we have for
$\xi \in \mR^d$ that
$$
H_K(\xi)= \sup_{x\in K} \langle x,\xi\rangle\leq  \sup_{x\in \epsilon^{-1}B} \langle x,\xi\rangle
= \epsilon^{-1}\sup_{y\in B} \langle y,\xi\rangle=\epsilon^{-1}H(\xi).
$$
Thus with $M:=   \epsilon^{-1}$, we have 
\begin{eqnarray*}
|\widehat{f}(z)|&\leq& C(1+\|z\|^2)^{N}
e^{MH(\textrm{Im}(z))}=Ce^{N\log(1+\|z\|^2)+M
  H(\textrm{Im}(z))}\\
&\leq& C e^{\max\{N,M\}p(z)}.
\end{eqnarray*}
So $\widehat{f}\in A_p$. 

Conversely, if $F\in A_p$, then
$$
|F(z)|\leq C_1 e^{C_2 p(z)}=C_1 (1+\|z\|^2)^{C_2}e^{C_2 H(\textrm{Im}(z))}.
$$
But for $\xi \in \mR^d$ we have 
$$
C_2H(\xi)=  C_2 \sup_{x\in B} \langle x,\xi\rangle
=
\sup_{y\in C_2 B} \langle y,\xi\rangle=H_{C_2 B} (\xi).
$$
So by the Payley-Wiener-Schwartz theorem, there exists an $f\in
\calE'(\mR^d)$ such that $\widehat{f}=F$ and the support of $f$ is
contained in $C_2B\subset C$. Thus $F\in \widehat{\calE'(C)}$.

\medskip

\noindent (3) Let $\mZ_{+}=\{0,1,2,3,\dots\}$. Let 
$$
Q(z)=\sum_{k\in \mZ_{+}^n,\;|k| \leq N} a_k
z^k,
$$
where for a multi-index $k=(k_1,\dots, k_n)\in \mZ_{+}^n$, 
$$
| k|:= k_1+\dots +k_n,\quad  z^k= z_1^{k_1}\dots z_n^{k_n},\quad 
\textrm{and}\quad a_k\in \mC.
$$
Consider
$$
q=\sum_{k\in \mZ_{+}^n,\;|k| \leq N} a_k\frac{1}{i^{|k|}} \frac{\partial^{|k|}}{\partial
  x_1^{k_1}\dots \partial x_n^{k_n}} \delta\in \calE'(C),
$$
Then $Q=\widehat{q}\in \widehat{\calE'(C)}=A_p$, and so $Q\in A_p$. 

 \medskip

 \noindent (4) Let $K_1$ and $K_2$ be nonnegative, and let $z,\zeta $
 satisfy $\|z-\zeta\|\leq e^{-K_1 p(z)-K_2}$. Then 
$$
\|z-\zeta\|\leq e^{-K_1 p(z)-K_2}=e^{-K_1 p(z)}e^{-K_2}
\leq 1\cdot 1=1.
$$
In particular, $\|\zeta\|\leq \|z\|+1$. Also, 
\begin{eqnarray*}
H(\textrm{Im}(\zeta-z))
&=&
\sup_{x\in B}\langle x, \textrm{Im}(\zeta-z)\rangle 
\leq \sup_{x\in B}\|x\| \| \textrm{Im}(\zeta-z)\|
\\
&\leq& \sup_{x\in B}\|x\|\|\zeta-z\|
\leq 1\cdot 1=1.
\end{eqnarray*}
Thus
\begin{eqnarray*}
p(\zeta)&=&\log (1+\|\zeta\|^2) +H(\textrm{Im}(\zeta))
\leq 2\log(1+\|\zeta\|)+H(\textrm{Im}(z+\zeta-z))
\\
&\leq& 2\log(2+\|z\|)+H(\textrm{Im}(z)) + H(\textrm{Im}(\zeta-z))
\\
&\leq&  \log (8(1+|z|^2))+H(\textrm{Im}(z))+1=p(z)+\log 8+1.
\end{eqnarray*}
This completes the proof. 
\end{proof}

\begin{proof}[Proof of Theorem~\ref{theorem_1}] Necessity of the condition
  \eqref{CC} is not hard to check. Indeed, if there are $g_1, \dots,
  g_n\in \calE'(C)$ such that
$$
f_1 \ast g_1 + \dots + f_n \ast g_n =\delta,
$$
then upon taking Fourier-Laplace transforms, we obtain 
$$
\widehat{f}_1(z)  \widehat{g}_1(z) + \dots + \widehat{f}_n(z) \widehat{g}_n(z) =1 \quad (z\in \mC^d).
$$
By the triangle inequality, 
$$
1= |\widehat{f}_1(z)  \widehat{g}_1(z) + \dots + \widehat{f}_n(z) \widehat{g}_n(z)|
\leq 
|\widehat{f}_1(z)||  \widehat{g}_1(z)| + \dots + |\widehat{f}_n(z) ||\widehat{g}_n(z)|.
$$
Suppose that $g_k$ has support contained in the compact convex set
$L_k$ ($\subset C$), where $k=1,\dots, n$. Then by the
Payley-Wiener-Schwartz theorem, we have an estimate
$$
|\widehat{g}_k(z)|\leq C_k (1+\|z\|^2)^{N_k} e^{H_{L_k}  (\textrm{Im}(z))}
$$ 
for each $k$. Let $\epsilon>0$ be small enough so that $\epsilon L_k
\subset B$ for all the $k$. Then we have for $\xi \in \mR^d$ that
$$
H_{L_k}(\xi)= \sup_{x\in L_k} \langle x,\xi\rangle\leq  \sup_{x\in \epsilon^{-1}B} \langle x,\xi\rangle
= \epsilon^{-1}\sup_{y\in B} \langle y,\xi\rangle=\epsilon^{-1}H(\xi).
$$
Thus we have that for all $k$,
$$
|\widehat{g}_k(z)|\leq C (1+\|z\|^2)^{N} e^{MH  (\textrm{Im}(z))},
$$
where $M:=\epsilon^{-1}$,  $C:=\displaystyle \max_k C_k$ and  $N:=\displaystyle \max_k N_k$. 
Consequently, 
$$
1\leq (|\widehat{f}_1(z)|+ \dots + |\widehat{f}_n(z)|) C
(1+\|z\|^2)^{N} e^{M H  (\textrm{Im}(z)) },
$$
and this yields \eqref{CC}, completing the proof of the necessity
part.

We now show the sufficiency of \eqref{CC}. Let $f_1,\dots, f_n\in
\calE'(C)$ be such that their Fourier-Laplace transforms satisfy
\eqref{CC}. Then by Lemma~\ref{lemma_tech_1}, $\widehat{f}_1,\dots,
\widehat{f}_n\in A_p$ with $p(z)=\log (1+\|z\|^2)+H(\textrm{Im}(z))$
($z\in \mC^d$). Moreover, this $p$ satisfies the conditions (1) and (2) of
Proposition~\ref{prop_Hor}.  The condition \eqref{CC} gives
\begin{eqnarray*}
|\widehat{f}_1(z)|+\dots+|\widehat{f}_n(z)|&\geq& 
C (1+ \|z\|^2)^{-N} e^{-MH(\textrm{Im}(z) )}
\\
&\geq& 
Ce^{-N\log(1+\|z\|^2) -MH(\textrm{Im}(z) )}
\\
&\geq& 
Ce^{-\max\{N,M\}p(z)}.
\end{eqnarray*}
It then follows from Proposition~\ref{prop_Hor} that there are some
$G_1, \dots, G_n$ in $A_p$ such that $ \widehat{f}_1 G_1+\dots +
\widehat{f}_n G_n =1$ on $\mC$. But $A_p=\widehat{\calE'(C)}$.  Hence
there exist $g_1,\dots, g_n\in \calE'(C)$ such that $f_1\ast
g_1+\dots+f_n \ast g_n=\delta$.
\end{proof}

\section{Special cases of the main result}

\subsection{The full space $\mR^d$} 

The supporting function $H$ of the unit ball $B$ in $\mR^d$ is given
by $H(\xi)=\|\xi\| $. So we obtain the following consequence of
Theorem~\ref{theorem_1}.

\begin{corollary}
Let $f_1, \dots, f_n\in \calE'(\mR^d)$. There exist $g_1,\dots,
g_n\in\calE'(\mR^d)$ such that
$$
f_1 \ast g_1 + \dots + f_n \ast g_n =\delta
$$
iff there are positive constants $C,N , M$ such that 
\begin{equation}
\label{CCC}
\textrm{for all }z\in\mC^d\textrm{, } 
|\widehat{f}_1(z)|+\dots+|\widehat{f}_n(z)|\geq 
C (1+ \|z\|^2)^{-N} e^{-M \|\textrm{\em Im}(z) \|}.
\end{equation} 
\end{corollary}

\subsection{The nonnegative orthant in $\mR^d$}

Let 
$$
\mR_+^d=\{x=(x_1,\dots, x_d)\in \mR^d: x_k\geq 0, \textrm{ for
  all }k=1,\dots, d\}.
$$  
The supporting function $H$ of $
B=\{x\in \mR_+^d:\|x\|\leq 1\}$ in $\mR^d$ is given by 
$$
H(\xi)=\|\xi^+\| ,
$$ 
where 
$$ 
\xi^+:=(\max\{\xi_1,0),\dots, \max\{\xi_d,0\}) 
$$
for $\xi=(\xi_1,\dots, \xi_d)\in \mR^d$.  Theorem~\ref{theorem_1} gives the following.

\begin{corollary}
Let $f_1, \dots, f_n\in \calE'(\mR_+^d)$. There exist $g_1,\dots,
g_n\in\calE'(\mR_+^d)$ such that
$$
f_1 \ast g_1 + \dots + f_n \ast g_n =\delta
$$
iff there are positive constants $C,N , M$ such that 
\begin{equation}
\label{CCCC}
\textrm{for all }z\in\mC^d\textrm{, } 
|\widehat{f}_1(z)|+\dots+|\widehat{f}_n(z)|\geq 
C (1+ \|z\|^2)^{-N} e^{-M \|(\textrm{\em Im}(z))^+ \|}.
\end{equation} 
\end{corollary}

In particular, in the case when $d=1$, we obtain:

\begin{corollary}
Let $f_1, \dots, f_n\in \calE'(\mR_+)$. There exist $g_1,\dots,
g_n\in\calE'(\mR_+)$ such that
$$
f_1 \ast g_1 + \dots + f_n \ast g_n =\delta
$$
iff there are positive constants $C,N , M$ such that 
\begin{equation}
\label{CCCC'}
\textrm{for all }z\in\mC\textrm{, } 
|\widehat{f}_1(z)|+\dots+|\widehat{f}_n(z)|\geq 
C (1+ |z|^2)^{-N} e^{-M \max\{\textrm{\em Im}(z),0\}}.
\end{equation} 
\end{corollary}

\subsection{The future light cone in $\mR^{d+1}$}

Let $C$ be the future light cone, namely,
$$
\Gamma:=\{(x,t)\in \mR^{d}\times \mR: \|x\|\leq c t, \; t \geq 0\},
$$
where $c$ denotes the speed of light.  Then the supporting function of
the intersection of $\Gamma$ and the unit ball in $\mathbb R^{d+1}$ is
given by
$$
\Phi(\xi,\tau)=\left\{ \begin{array}{ll}
\displaystyle \sqrt{\|\xi\|^2+\tau^2}              & \textrm{if } c^{-1} \|\xi\| \leq \tau, \\
\displaystyle\frac{\tau+c\|\xi\|}{\sqrt{c^2+1}}    & \textrm{if } -c\|\xi\|\leq \tau \leq c^{-1}\|\xi\|,\\
0                                                  & \textrm{if } \tau \leq -c\|\xi\|,
\end{array}\right.
$$
for $(\xi,\tau)\in \mR^d\times \mR$.
Then we have:

\begin{corollary}
Let $f_1, \dots, f_n\in \calE'(\Gamma)$. There exist $g_1,\dots,
g_n\in\calE'(\Gamma)$ such that
$$
f_1 \ast g_1 + \dots + f_n \ast g_n =\delta
$$
iff there are positive constants $C,N , M$ such that 
\begin{equation}
\label{CCCC}
\textrm{for all }z\in\mC^d\textrm{, } 
|\widehat{f}_1(z)|+\dots+|\widehat{f}_n(z)|\geq 
C (1+ \|z\|^2)^{-N} e^{-M \Phi(\textrm{\em Im}(z))}.
\end{equation} 
\end{corollary}

\section{Answer to Yamamoto's question}

We remark that Theorem~\ref{theorem_1} answers an open question of
Y.~Yamamoto; see question number 2 \cite[p.282]{Yam}. There it was
asked if for $f_1,f_2 \in \calE'(\mR)$, the condition that
$\widehat{f}_1,\widehat{f}_2$ have no common zeros in $\mC$ is enough
to guarantee that there are $g_1,g_2\in \calE'(\mR)$ such that
$f_1\ast g_1+f_2\ast g_2=\delta$. 

In light of Theorem~\ref{theorem_1} above, the answer is no, since our
analytic condition \eqref{CCC} (in the case when $d=1$) is not
equivalent to (and is stronger) than the condition that there is no
common zero, as seen in the following example.  (The idea behind this
example is taken from \cite{Mei}.)

\begin{example}
Let $c\in \mR_+$ be the {\em Liouville constant}, that is,
$$
c=\displaystyle \sum_{n=1}^{\infty}\frac{1}{10^{n!}}.
$$
(See for example, \cite{Apo}.) 
Then it can be seen that $c$ is irrational. Also, for $K\in \mN$, with
$p_K$, $q_K$ defined by
$$
p_K=10^{K!} \sum_{k=1}^K \frac{1}{10^{k!}},\quad q_{K}=10^{K!},
$$
we have that 
\begin{eqnarray}
  \nonumber 
  0<\left| c-\frac{p_K}{q_K}\right| 
  &=& \sum_{k=K+1}^\infty  \frac{1}{10^{k!}}
  = \frac{1}{10^{(K+1)!}}+\frac{1}{10^{(K+2)!}}+\frac{1}{10^{(K+3)!}}+\dots
  \\\nonumber
  &\leq & \frac{1}{10^{(K+1)!}}\cdot \sum_{m=0}^\infty \frac{1}{10^m}
  = \frac{1}{10^{(K+1)!}} \cdot \frac{10}{9} 
  =\frac{10/9}{(10^{K!})^{K} 10^{K!}}
  \\
\label{L}
&\leq &\frac{1}{(10^{K!})^K}=\frac{1}{q_K^K}.
\end{eqnarray}
Take $f_1=\delta-\delta_{c}$ and $f_2=\mathbf{1_{[0,1]}}$, where
$\mathbf{1_{[0,1]}}$ denotes the indicator function of the interval
$[0,1]$.  Then $f_1, f_2$ belong to $\calE'(\mR)$ and we have that
$$
\widehat{f}_1(z)=1-e^{-icz},
\quad
\widehat{f}_2(z)=\left\{\begin{array}{ll}
\displaystyle\frac{e^{-iz}-1}{-iz}& \textrm{if }z\neq 0\\
\displaystyle i \frac{de^{-iz}}{dz}\bigg|_{z=0}=1 &\textrm{if }z= 0.\end{array} \right. 
$$
Then $\widehat{f}_1$ and $\widehat{f}_2$ have no common zeros
(otherwise $c$ would be rational!). We now show that \eqref{CCC} does
not hold.  Suppose, on the contrary that there exist $C,N,M$ positive
such that
\begin{equation}
\label{eq_CCCC}
|\widehat{f}_1(z)|+|\widehat{f}_2(z)|\geq C(1+|z|^2)^{-N}e^{-M |\textrm{Im}(z)|}
\end{equation}
for all $z\in \mC$. If $z=2\pi q_K$, then 
we have $\widehat{f}_2(2\pi q_K)=0$. On the other hand,
$$
|\widehat{f}_1(2\pi q_K)|=|1-e^{-ic(2\pi q_K)}|= |\sin (\pi c q_K)|=|\sin (\pi cq_K-\pi p_K)|.
$$
The inequality \eqref{eq_CCCC} now yields that 
$$
|\sin (\pi (cq_K-p_K))|\geq C(1+4\pi^2 q_K^2)^{-N}.
$$
But $|\sin\Theta |\leq|\Theta|$ for all real $\Theta$, and so we obtain 
$$
\pi q_K \left|c-\frac{p_K}{q_K}\right| \geq C ( 1+4\pi^2 q_K^2)^{-N}. 
$$
In light of  \eqref{L}, we now obtain 
$$
\pi q_K \frac{1}{q_K^K} \geq C ( 1+4\pi^2 q_K^2)^{-N}, 
$$
and rearranging, we have 
$$
\pi \frac{ q_K  (1+ 4\pi^2 q_K)^{N}}{q_K^K}\geq C.
$$
Passing the limit $K\rightarrow \infty$, we arrive at the
contradiction that $0 \geq C $.

We remark that in this example $f_1,f_2$ actually belong to
$\calE'(\mR_+)$, and with the same argument given above, it can be
seen that $\widehat{f}_1, \widehat{f}_2$ don't satisfy \eqref{CCCC'}
either.  This also gives another example answering question number 1
in \cite{Yam}, namely, for $f_1,f_2$ in $\calE'(\mR_+)$, whether the
condition that $\widehat{f}_1,\widehat{f}_2$ have no common zeros is
enough to guarantee that there are $g_1,g_2\in \calE'(\mR_+)$ such
that $f_1\ast g_1+f_2\ast g_2=\delta$.
\end{example}

\end{document}